\documentclass[a4paper,twoside,10pt,reqno]{article}
\usepackage{graphicx}
\usepackage{epstopdf}
\usepackage[T1]{fontenc}
\usepackage[latin1]{inputenc}
\usepackage{amssymb,amsmath,amsthm,dsfont,mathrsfs}
\usepackage{amsfonts,euscript}
\usepackage{color}
\usepackage{authblk}
\usepackage{multirow}
\usepackage{mathtools}

\usepackage{authblk}
\usepackage{fullpage}
\usepackage{setspace}
\newtheorem{defi}{Definition}[section]
\newtheorem{teo}{Theorem}[section]
\newtheorem{pro}[teo]{Proposition}

\newcommand{\nid}{\noindent }

\begin{document}

	\title{A crossinggram for random fields on lattices}
	
	\author{Helena Ferreira}
	\affil{Universidade da Beira Interior, Centro de Matem\'{a}tica e Aplica\c{c}\~oes (CMA-UBI), Avenida Marqu\^es d'Avila e Bolama, 6200-001 Covilh\~a, Portugal, \texttt{helenaf@ubi.pt}}

	\author{Marta Ferreira}
	\affil{Center of Mathematics of Minho University\\ Center for Computational and Stochastic Mathematics of University of Lisbon \\
		Center of Statistics and Applications of University of Lisbon, Portugal, \texttt{msferreira@math.uminho.pt} }

	\author{Lu\'is A. Alexandre}
	\affil{Universidade da Beira Interior\\NOVA LINCS\\6200-001 Covilh\~a, Portugal, \texttt{luis.alexandre@ubi.pt}}
	
	\date{}
	
	\maketitle
	
\abstract{
The modeling of risk situations that occur in a space-time framework can be done using max-stable
random fields on lattices.
Although the summary coefficients for  the spatial and temporal behaviour do not characterize the
finite-dimensional distributions of the random field, they have the advantage of being immediate to
interpret and easier to estimate.
The coefficients that we propose give us information about the tendency of a random field
for local oscillations of its values in relation to real valued high levels. It is not the magnitude of the oscillations that is being evaluated, but rather the greater or lesser number of oscillations, that is, the tendency of the trajectories to oscillate.
We can observe surface
trajectories more smooth over a region according to higher crossinggram value.
It takes value in $[0,1]$ and increases with
the concordance of the variables of the random field.
}\\

\graphicspath{{./Figuras/}}

\nid\textbf{keywords:} {extreme values, upcrossings, tail dependence coefficients, extremal coefficients}

\nid\textbf{AMS 2000 Subject Classification}: 60G70

\section{Introduction}\label{sintro}
The modeling of risk situations that occur in a space-time framework can be done using max-stable random fields on lattices.
Consider that $Y_i(x)$ represents the daily maximum precipitation in year $i$ at a location $x$ belonging to some locations family $A\subset \mathbb{Z}^2$. The stochastic behavior of $\{Y_n(x),\, x\in A,\, n\geq 1\}$ can not be studied using the classical theory of stable distributions because the variables of interest are not sums, thus excluding any modeling with normal multivariate distributions (Embrechts \emph{et al.}~\cite{Embrechts1997} 1997). If we are interested in assessing probabilities of risk events, such as  ``the maximum, in a region A, of the maximum daily rainfall  over $n$ years exceeds $u$", $P(\bigvee_{x \in A}\bigvee_{i=1}^{n}Y_i(x)>u)$, with notation $a \vee b=\max(a,b)$, we have to use a theory that provides information about the distributions of variables $\bigvee_{i=1}^{n}Y_i(x)$, $x\in A$, and the dependency structure between them, i.e., the theory of multivariate extreme values distributions (Ribatet \emph{et al.}~\cite{Ribatet2016} 2016). In the context of this theory, it is considered that, as $n\to \infty$, the set of approximate distributions for $\bigvee_{i=1}^{n}Y_i(x)$ admits only Fr\'echet, Weibull and Gumbel laws and that the approximate dependence function $D_{x_1,...,x_d}$ for vector $\left(\bigvee_{i=1}^{n}Y_i(x_1),...,\bigvee_{i=1}^{n}Y_i(x_d)\right)$, whatever the choice
of locations $(x_1,...,x_d)$, is max-stable, i.e., satisfies the condition 
\begin{eqnarray}\nonumber
D_{x_1,...,x_d}^k(u_1,...,u_d)=D_{x_1,...,x_d}(u_1^k,...,u_d^k),\,\forall k>0,\,u_i\in [0,1],
\end{eqnarray}
(de Haan and Ferreira \cite{Haan2006} 2006). \\
This will be the main context of this work: we consider random fields $\{X(x),\, x\in \mathbb{Z}^2\}$, for which $(X(x_1),...,X(x_d))$ has multivariate extreme values distribution, regardless the choice of the locations vector $(x_1,...,x_d)$. Its distribution function is completely characterized by the marginal laws and by its exponent function. A widely used choice for marginal distributions is the unit Fr\'echet, for sake of simplicity and without loss of generality. The exponent function 
\begin{eqnarray}\nonumber
\ell_{x_1,...,x_d}(z_1,...,z_d)=-\log D_{x_1,...,x_d}(P(X(x_1)\leq z_1),...,P(X(x_d)\leq z_d)),\,\forall z_i\geq 0,
\end{eqnarray}
verifies
\begin{eqnarray}\label{ell}
\ell_{x_1,...,x_d}(z_1,...,z_d)=\frac{E\left(\bigvee_{i=1}^{d}U(x_i)^{z_i}\right)}
{1-E\left(\bigvee_{i=1}^{d}U(x_i)^{z_i}\right)},\,\forall z_i\geq 0,
\end{eqnarray}
where $(U(x_1),...,U(x_d))$ is a vector of standard uniform distributed marginals having the same dependence function $D_{x_1,...,x_d}$.\\
The estimation of the exponent function presents challenges (see, e.g., Beirlant \emph{et al.}~\cite{Beirlant2004} 2004, Ferreira and Ferreira \cite{Ferreira2012a} 2012a, Beirlant \emph{et al.}~\cite{Beirlant2016} 2016, Escobar \emph{et al.}~\cite{Escobar2018} 2018, Kiriliouk \emph{et al.}~\cite{Kiriliouk2018} 2018 and references therein) and several summary measures of the dependence between the variables of a max-stable random field can be used: extremal coefficients (Tiago de Oliveira \cite{Tiago1962} 1962/1963, Smith \cite{Smith1990}  1990), coefficients of tail dependence (Sibuya \cite{Sibuya1960} 1960, Joe \cite{Joe1997} 1997, Li \cite{Li2009} 2009), coefficients of pre-asymptotic tail dependence (Ledford and Tawn \cite{Ledford1997} 1997, Wadsworth and Tawn \cite{Wadsworth2012}-\cite{Wadsworth2013} 2012-2013), fragility coefficients (Falk and Tichy \cite{F+T} 2011, Ferreira and Ferreira \cite{Ferreira2012b} 2012b), madogram (Naveau \emph{et al.}~\cite{Naveau2009} 2009, Ferreira and Ferreira \cite{Ferreira2018} 2018), extremogram (Davis and Mikosch \cite{D+M2009} 2009), among others.\\
Although the summary coefficients of the spatial dependence structure do not characterize the finite-dimensional distributions of $\{X(x),\,x\in \mathbb{Z}^2\}$, they have the advantage of being immediate to interpret and easier to estimate. The coefficients that we propose, study and apply here give us information about the tendency of $\{X(x),\,x\in A\}$, $A\subset \mathbb{Z}^2$, for local oscillations of their values in relation to real high levels $u$. We can observe  trajectories $\{x,X(x)\}_{x\in A}$ more or less smooth (or more or less rough) according to the coefficients values.\\ 
The  tendency for the variables, in close locations, to jointly present extreme values will determine  the proportion of exceedances of  high levels that are upcrossings of the level. As the joint tendency for extreme values is usually summarized in the literature by upper-tail dependence coefficients, the question arises:
how to use these coefficients to summarize the degree of this kind of smoothness for a random field on a lattice?
We invite the reader to follow us in a motivated and justified construction of a response to this question.\\
The objective of this work is to quantify the propensity of a max-stable random field for oscillations over regions $A\subset \mathbb{Z}^2$,  through coefficients that are easy to estimate and use in applications. Thus, in the next sections, we will introduce the crossinggram $\zeta(A)$, $A\subset \mathbb{Z}^2$ , we will deduce some of their properties and propose a method for its estimation. We will illustrate the calculation of $\zeta(A)$ in a model for max-stable random fields. Section \ref{snonmaxstab} is concerned with oscillations of general random fields, for which we propose a smaller  coefficient easier to deal with, but less interesting for the max-stable context. 

\section{Notations and construction of the crossinggram}\label{scoef}

Let $\{X(x),\, x\in \mathbb{Z}^2\}$ be a max-stable random field, i.e., the variables $X(x)$ have extreme-type distribution and, for any choice of locations $x_1,\,...,\, x_d$, the vector $(X(x_1),...,X(x_d))$ has multivariate extreme values distribution. Without loss of generality for applications, suppose that $X(x)$ has common distribution function (d.f.) unit Fr\'echet, i.e., $F(x)=\exp(-x^{-1})$, $x>0$. \\
For each location $x=\left(x^{(1)},x^{(2)}\right)\in \mathbb{Z}^2$, let $V_d(x):=\left\{y\in  \mathbb{Z}^2:\, \left\Arrowvert y-x\right\Arrowvert\leq d\right\}$. For the particular case of $d = 1$ that we will highlight, we simply write $V(x)$. \\
We say that $\{X(x),\, x\in \mathbb{Z}^2\}$ has an oscillation with respect to $u$, $u\in(0,1)$, at location $x$, when the following event occurs
\begin{eqnarray}\nonumber
\left\{F(X(x))\leq u < \bigvee_{y\in V(x)} F(X(y))\right\}\,,
\end{eqnarray}
and it has an exceedance of $u$, at location $x$, when $\left\{F(X(x))> u \right\}$ occurs.\\
Several tail dependence coefficients for bivariate and multivariate distributions have been constructed in the literature and, in our view, the work of Li (\cite{Li2009}, 2009) is an important landmark. For our purpose, we take as a good starting point the upper-tail dependence of $\bigvee_{y\in V(x)}F(X(y))$ and $\bigvee_{y\in A}F(X(y))$, for each location $x\in A \subset \mathbb{Z}^2$, where the region $A$ is bounded and its finite cardinal will be denoted $|A|$.\\ 
Consider, for some location $x\in A$,
\begin{eqnarray}\nonumber
\lambda(V(x)|A)=\displaystyle \lim_{u\uparrow 1}P\left(\bigvee_{y\in V(x)}F(X(y))>u\bigg |\bigvee_{y\in A}F(X(y))>u\right)
\end{eqnarray}
and
\begin{eqnarray}\nonumber
\lambda(x|A)=\displaystyle \lim_{u\uparrow 1}P\left(F(X(x))>u\bigg|\bigvee_{y\in A}F(X(y))>u\right)\,.
\end{eqnarray}
We intuitively expect smaller values for the difference
\begin{eqnarray}\nonumber
\sum_{x \in A}\lambda(V(x)|A)-\sum_{x \in A}\lambda(x|A)
\end{eqnarray}
in regions where, for each $x$, the variables $F(X(y))$, $y\in V(x)$, have lower tendency for oscillations relative to high levels. \\
The following proposition justifies this interpretation for the values of these differences, presenting them as coefficients that summarize the expected number of local oscillations in  the spatial  context.\\
\begin{pro}
For $x\in A \subset \mathbb{Z}^2$, the tail dependence coefficients $\lambda(V(x)|A)$ and $\lambda(x|A)$ satisfy
$$\displaystyle\sum_{x \in A}\lambda(V(x)|A)-\sum_{x \in A}\lambda(x|A)=\displaystyle \lim_{u\uparrow 1}E\left(\sum_{x \in A}\mathbf{1}_{\{F(X(x))\leq u<{\bigvee}_{y\in V(x)}F(X(y))\}}\bigg|\sum_{x \in A}\mathbf{1}_{\{F(X(x))>u\}}>0\right).$$
\end{pro}

\begin{proof}
Observe that 
\begin{eqnarray}\nonumber
\begin{array}{rl}
&\displaystyle\sum_{x \in A}\lambda(V(x)|A)-\sum_{x \in A}\lambda(x|A)
= \displaystyle \lim_{u\uparrow 1}\frac{\displaystyle\sum_{x \in A}P\left(\displaystyle\mathop{\bigvee}_{y\in V(x)}F(X(y))>u\right)-\sum_{x \in A}P(F(X(x))>u)}{P\left(\displaystyle\mathop{\bigvee}_{x\in A}F(X(x))>u\right)}\\\\
=& \displaystyle \lim_{u\uparrow 1}\frac{\displaystyle\sum_{x \in A}P\left(F(X(x))\leq u < \displaystyle\mathop{\bigvee}_{y\in V(x)}F(X(y))\right)}{P\left(\displaystyle\mathop{\bigvee}_{x\in A}F(X(x))>u\right)}\,.
\end{array}
\end{eqnarray}
\end{proof}
We depart from the above representation and, with a convenient normalization in order to eliminate the effect of the dimension $|A|$, we propose coefficients $\zeta(A)$, $A\subset \mathbb{Z}^2$, with values in $[0, 1]$. For the normalization here we take into account that $P	\left(F(X(x))\leq u<\bigvee_{y\in V(x)}F(X(y))\right)\leq \sum_{y\in V(x)}P(F(X(y))>u)$, which will lead to a usefull representation of the coefficient in Proposition \ref{pSAepsilon}.\\
The following crossinggram increases when the local oscillations decrease and, if we consider "roughness" the proximity between the number of upcrossings and exceedances of high levels we can say that crossinggram increases with the local "smoothness" of the random field.\\
\begin{defi}\label{dSASnm}
The crossinggram $\zeta(A)$, $A\subset \mathbb{Z}^2$, for the max-stable random field $\{X(x),\,x \in \mathbb{Z}^2\}$ is defined by
\begin{eqnarray}\nonumber\label{SAlimE}
\zeta(A)=1-\displaystyle\lim_{u\uparrow 1}\frac{\displaystyle E\left(\sum_{x \in A}\mathbf{1}_{\{F(X(x))\leq u<{\bigvee}_{y\in V(x)}U(y)\}}\bigg|\sum_{x \in A}\mathbf{1}_{\{F(X(x))>u\}}>0\right)}{\displaystyle E\left(\displaystyle\sum_{x \in A}\mathop{\sum}_{y\in V(x)-\{x\}}\mathbf{1}_{\{F(X(y))>u\}}\bigg|\sum_{x \in A}\mathbf{1}_{\{F(X(x))>u\}}>0\right)}\,.
\end{eqnarray}
\end{defi}
As will be pointed in the next section, the limits $\zeta(A)$, $A\subset \mathbb{Z}^2$, exist and have enlightening properties. 

\section{Properties of the crossinggram}\label{sproperties}

 The coefficients of tail dependence can be related to the extremal coefficients (see, e.g., Beirlant \emph{et al.}~\cite{Beirlant2004} 2014). We remind that, for any $x_1,...,x_d\in \mathbb{Z}^2$, we have
\begin{eqnarray}\label{ce}
P\left(\bigcap_{i=1}^{d}\{F(X(x_i))\leq u\}\right)=u^{\theta(x_1,...,x_d)},
\end{eqnarray}
with $\theta(x_1,...,x_d)$ constant in $[1,d]$ and $\theta(x_1,...,x_d)=\ell_{x_1,...,x_d}(1,...,1)$. In the case of $\theta(x_i,\,i\in I)$ we simply write $\theta(I)$.\\
In the particular case of $d=2$ and isotropic stationary max-stable random fields, we can consider the extremal function 
\begin{eqnarray}\nonumber
\underline{\theta}(h)\equiv \theta(x,x+h)=E(X(x)\vee X(x+h)),
\end{eqnarray}
since the dependence between $X(x)$ and $X(y)$ will only depend on the distance between $x$ and $y$. For some models of continuos max-stable random fields found in literature (Smith, Schlather, Brown-Resnick, Extremal-t), an expression for  $\underline{\theta}(h)$ is available. \\
The coefficients of tail dependence and the extremal coefficients can be considered dual when we study their variation with the concordance of the variables: when concordance increases, the bivariate upper-tail dependence rises and the extremal coefficients fall. The proposed $\zeta(A)$  coefficients increase with increasing local concordance of the random field variables, as can easily be seen if we express them from the extremal coefficients. Before we establish the properties that justify the utility and interpretation of the proposed crossinggram, we first present a representation for it through the extremal coefficients, which will also motivate their estimation.\\
Observe that 
\begin{eqnarray}\nonumber
\lambda(V(x)|A)=\displaystyle\lim_{u\uparrow 1}\frac{P\left(\bigvee_{y\in V(x)}U(y)>u\right)}{P\left(\bigvee_{y\in A}U(y)>u\right)}=\displaystyle\lim_{u\uparrow 1}\frac{1-u^{\theta(V(x))}}{1-u^{\theta(A)}}=\frac{\theta(V(x))}{\theta(A)}.
\end{eqnarray}
Simple calculations allow then to obtain the following representation for the crossinggram, from the proposition 2.1.\\

\begin{pro}\label{pSAepsilon}
The crossinggram $\zeta(A)$, $A\subset \mathbb{Z}^2$, for the max-stable random field $\{X(x),\,x \in \mathbb{Z}^2\}$ satisfies 
\begin{eqnarray}\label{SA}
\zeta(A)=\frac{\mathcal{V}(A)-\sum_{x \in A}\theta(V(x))}{\mathcal{V}(A)-|A|},
\end{eqnarray}	
where $\mathcal{V}(A)=\sum_{x \in A}|V(x)|$.
\end{pro}
\vspace{0,5cm}

If $\{X(x),\,x\in \mathbb{Z}^2\}$ is isotropic, stationary and all $V(x)$, $x\in A$, have the same shape, then
\begin{eqnarray}\nonumber
\zeta(A)=\frac{\mathcal{V}(A)-|A|\theta(V(x_0))}{\mathcal{V}(A)-|A|},
\end{eqnarray}	
for some $x_0 \in A$. \\
Several extremal coefficients type-functionals can be defined as indicators of dependence of the variables over A but this is not de purpose of the proposition 3.1. For instance, $\gamma(A)=\frac{\theta(A)-1}{|A|-1}$ captures the overall dependence of the variables $X_i$, $i\in A$, but does not incorporate the discrepancy between local dependences neither the propensity for local upcrossings. The expression (\ref{SA}) is a usefull representation for the crossinggram. The insight to $\zeta(A)$  is not available from this representation rather from its definition.\\
We will apply the above proposition to derive the next properties and propose an estimator, beyond the moment estimator suggested from the definition.\\

\begin{pro}
	The crossinggram $\zeta(A)$, $A\subset \mathbb{Z}^2$, for the max-stable random field $\{X(x),\,x \in \mathbb{Z}^2\}$ satisfies 
	\begin{itemize}
		\item[(i)]  
		$\zeta(A)\in [0,1]$.
		\item[(ii)] $\zeta(A)=0$ if and only if the variables of $\{X(x),\,x\in A\}$ are independent;
		\item[(iii)]$\zeta(A)=1$ if and only if the variables of $\{X(x),\,x\in A\}$ are totally dependent;
		\item[(iv)] $\zeta(A)$ increases with the concordance between the variables of $\{X(x),\,x\in A\}$.
	\end{itemize}
\end{pro}
\begin{proof}
(i) The statement results from $1\leq \theta(V(x))\leq |V(x)|$, $\forall x\in A$.\\
(ii) $\zeta(A)=0\Leftrightarrow \theta(V(x))=|V(x)|$, $\forall x\in A$, which occurs if and only if, variables $X(x)$, $x\in A$, are independent. \\
(iii) $\zeta(A)=1\Leftrightarrow \theta(V(x))=1$, $\forall x\in A$, which occurs if and only if, variables $X(x)$, $x\in A$, are  totally dependent. \\
(iv) Suppose that the variables of $\{Y(x),\,x\in A\}$ are more concordant than those of $\{X(x),\,x\in A\}$. This means that, for any $z(x)\in [0,1]$, with $x\in A$,
\begin{eqnarray}\nonumber
P\left(\bigcap_{x\in A}\{F(Y(x))>z(x)\}\right)\geq P\left(\bigcap_{x\in A}\{F(X(x))>z(x)\}\right)
\end{eqnarray}
and
\begin{eqnarray}\nonumber
P\left(\bigcap_{x\in A}\{F(Y(x))\leq z(x)\}\right)\geq P\left(\bigcap_{x\in A}\{F(X(x))\leq z(x)\}\right).
\end{eqnarray}
Then (Shaked and Shanthikumar \cite{Shaked2007} 2007) we have
\begin{eqnarray}\nonumber
E\left(\bigvee_{y\in V(x)}F(Y(y))\right)\leq E\left(\bigvee_{y\in V(x)}F(X(y))\right)
\end{eqnarray}
and
\begin{eqnarray}\nonumber
\frac{E\left(\bigvee_{y\in V(x)}F(Y(y))\right)}{1-E\left(\bigvee_{y\in V(x)}F(Y(y))\right)}\leq \frac{E\left(\bigvee_{y\in V(x)}F(X(y))\right)}{1-E\left(\bigvee_{y\in V(x)}F(X(y))\right)},
\end{eqnarray}
that is, by (\ref{ell}) and (\ref{ce}), $\theta^{(Y)}(V(x))\leq \theta^{(X)}(V(x))$, $\forall x\in A$. Thus, from the previous proposition it results $S^{(Y)}(A)\geq S^{(X)}(A)$, where the upper indexes distinguish the fields to which the coefficients refer. 
\end{proof}

\section{Estimation of $\zeta(A)$}\label{sestimation}

We recall that the extremal coefficient corresponds to the exponent function at the unit vector and thus, considering (\ref{ell}), we have
\begin{eqnarray}\label{epsilonellestim}
\theta(x_1,...,x_d)=\ell_{x_1,...,x_d}(1,...,1)
=\frac{E\left(\bigvee_{i=1}^{d}U(x_i)\right)}
{1-E\left(\bigvee_{i=1}^{d}U(x_i)\right)}=\frac{1}
{1-E\left(\bigvee_{i=1}^{d}U(x_i)\right)}-1,
\end{eqnarray}
where $U(x_i)=F(X(x_i))$, $i=1,\dots,d$. Ferreira and Ferreira (\cite{Ferreira2012a} 2012a) presented an estimator for $\theta(x_1,...,x_d)$ by taking the sample mean in place of the expected value. Strong consistency and asymptotic normality were also addressed in Ferreira and Ferreira (\cite{Ferreira2012a} 2012a). 
Here we follow the same methodology. More precisely, consider $\left\{X^{(j)}(x),\,x\in A\right\}$, $j=1,...,n$, a random sample coming from  $\left\{X(x),\,x\in A\right\}$. 
Based on (\ref{SA}) and (\ref{epsilonellestim}), we state
\begin{eqnarray}\label{estimSA}
\hat{S}(A)=\frac{\mathcal{V}(A)-\sum_{x \in A}\hat{\theta}(V(x))}{\mathcal{V}(A)-|A|},
\end{eqnarray}
where
\begin{eqnarray}\nonumber
\hat{\theta}(V(x))
=\frac{1}{1-\frac{1}{n}\sum_{j=1}^{n}\bigvee_{y\in V(x)}\hat{U}_j(y)}-1,
\end{eqnarray}
with
\begin{eqnarray}\nonumber
\hat{U}_j(y)=\frac{1}{n+1}\sum_{l=1}^{n}\mathbf{1}_{\{X^{(l)}(y)\leq X^{(j)}(y)\}}.
\end{eqnarray}

\section{The crossinggram outside the max-stable context}\label{snonmaxstab}
For a random field $\{X(x),\,x\in \mathbb{Z}^2\}$, not necessarily max-stable, we can define $\zeta(A)$ as previously
\begin{eqnarray}\nonumber
\zeta(A)=1-\displaystyle\lim_{u\uparrow 1}\frac{\displaystyle E\left(\sum_{x \in A}\mathbf{1}_{\{U(x)\leq u<{\bigvee}_{y\in V(x)}U(y)\}}\bigg|\sum_{x \in A}\mathbf{1}_{\{U(x)>u\}}>0\right)}{\displaystyle E\left(\displaystyle\sum_{x \in A}\mathop{\sum}_{y\in V(x)-\{x\}}\mathbf{1}_{\{U(y)>u\}}\bigg|\sum_{x \in A}\mathbf{1}_{\{U(x)>u\}}>0\right)}\,,
\end{eqnarray}
with $U(x)=F_{X(x)}(X(x))$, $x\in \mathbb{Z}^2$, provided the limit exists. We can consider different marginals and the relationship with the tail dependence coefficients remains valid. 
However, we don't have the relation between the tail dependence coefficients $\lambda$ and the extremal coefficients $\theta$, and therefore Proposition \ref{pSAepsilon} is not valid. 
Consequently, the estimation method proposed for $\zeta(A)$ can not be used. 
The estimation of the coefficients could be done through the moment estimation for the expectations in its definition, or estimation methods for tail dependence coefficients, already mentioned.\\
Since bivariate tail dependence coefficients can be more easily computed and estimated than multivariate tail dependence coefficients, we propose now a smaller measure $\zeta^*(A)\leq \zeta(A)$  dependent only on bivariate marginal distributions. Its main drawback is to not take into account joint exceedances of $u$ in $V(x)$ and, for isotropic and stationary random fields, it reduces to bivariate $\lambda$. Its advantages over $\zeta(A)$ are the availability of several models for bivariate tail dependence in the literature and a simpler estimation.

\begin{defi}
The crossinggram $\zeta^*(A)$, $A\subset \mathbb{Z}^2$, for the  random field $\{X(x),\,x \in \mathbb{Z}^2\}$ is defined by
\begin{eqnarray}\nonumber
\zeta^*(A)=1-\displaystyle\lim_{u\uparrow 1}\frac{\displaystyle E\left(\sum_{x \in A}\sum_{y\in V(x)-\{x\}}\mathbf{1}_{\{U(x)\leq u<U(y)\}}\bigg|\sum_{x \in A}\mathbf{1}_{\{U(x)>u\}}>0\right)}{\displaystyle E\left(\displaystyle\sum_{x \in A}\mathop{\sum}_{y\in V(x)-\{x\}}\mathbf{1}_{\{U(y)>u\}}\bigg|\sum_{x \in A}\mathbf{1}_{\{U(x)>u\}}>0\right)}\,,
\end{eqnarray}
provided the limit exists. 
\end{defi}

\section{Example}\label{sexamples}

Let $\{Y(x),\, x\in \mathbb{Z}^2\}$ be a  random field with independent variables and independent of the random variable $R$. Suppose that $F_{Y(x)}(z)=F_R(z)=\exp(-z^{-1})$, $z>0$, and that $\{\beta_1,\dots,\beta_k\}$ is a family of constants in $(0,1]$. \\
For a fixed partition $\mathcal{P}=\{A_1,\dots,A_k\}$ of $\mathbb{Z}^2$, we define 
\begin{eqnarray}\nonumber
X(x)=Y(x)\beta(x)\vee R(1-\beta(x)),\,x\in \mathbb{Z}^2,\, n\geq 1,
\end{eqnarray}
with $\beta(x)=\sum_{i=1}^{k}\beta_i\mathds{1}_{A_i}(x)$, $x\in \mathbb{Z}^2$.\\
We have $F_{X(x)}(z)=P(X(x)\leq z)=e^{-z^{-1}}$ 
and, for any choice of locations $x_1,\dots,x_d$,
\begin{eqnarray}\nonumber
F_{(X(x_1),\dots, X(x_d))}(z_1,\dots,z_d)=P\left(\bigcap_{j=1}^{d} \{X(x_j)\leq z_j\}\right)=\exp\left(-\sum_{j=1}^{d}z_j^{-1}\beta(x_j)-\bigvee_{j=1}^{d}z_j^{-1}(1-\beta(x_j))\right).
\end{eqnarray}
The dependence function of $(X(x_1),\dots,X(x_d))$ is given by
\begin{eqnarray}\label{exDepFun}
\begin{array}{rl}
\displaystyle D_{x_1,\dots,x_d}(u_1,\dots,u_d)=&\displaystyle\sum_{(i_1,\dots,i_d)\in \{1,\dots,k\}^d}\mathds{1}_{A_{i_1}\times\dots\times A_{i_d}}(x_1,\dots,x_d)\cdot \prod_{j=1}^{d}u_j^{\beta_{i_j}}\bigwedge_{j=1}^{d} u_j^{(1-\beta_{i_j})}\\\\
=&\displaystyle\prod_{j=1}^{d}u_j^{\beta(x_j)}\bigwedge_{j=1}^{d} u_j^{(1-\beta(x_j))},
\end{array}
\end{eqnarray}
which is max-stable. We remark that, if locations $x_1,\dots,x_d$ belong to the same region $A_s$, we have
\begin{eqnarray}\nonumber
\begin{array}{rl}
\displaystyle D_{x_1,\dots,x_d}(u_1,\dots,u_d)=\displaystyle\prod_{j=1}^{d}u_j^{\beta_s}\bigwedge_{j=1}^{d} u_j^{(1-\beta_s)}=\left(\prod_{j=1}^{d}u_j\right)^{\beta_s}\left(\bigwedge_{j=1}^{d} u_j\right)^{(1-\beta_s)},
\end{array}
\end{eqnarray}
which is a geometric mean of the product copula and the minimum copula. \\
In general, if $x_1\in A_{i_1},\dots,x_d\in A_{i_d}$, we have
\begin{eqnarray}\nonumber
\begin{array}{rl}
\displaystyle D_{x_1,\dots,x_d}(u_1,\dots,u_d)=\displaystyle D_{\prod}\left(u_1^{\beta_{i_1}},\dots,u_d^{\beta_{i_d}}\right)
D_{\bigwedge}\left(u_1^{1-\beta_{i_1}},\dots,u_d^{1-\beta_{i_d}}\right),
\end{array}
\end{eqnarray}
where $D_{\prod}$ and $D_{\bigwedge}$ respectively denote the copulas of vectors with independent and totally dependent marginals.\\
From (\ref{exDepFun}) we obtain, for $V(x)=\{x_1,\dots,x_d\}$,
\begin{eqnarray}\nonumber
\begin{array}{rl}
\displaystyle \theta(V(x))=\sum_{y\in V(x)}\beta(y)+\bigvee_{y\in V(x)}(1-\beta(y)).
\end{array}
\end{eqnarray}
In particular,
\begin{eqnarray}\nonumber
\displaystyle \theta(x_1,x_2)=
\left\{
\begin{array}{ll}
1+\beta_s & ,\,\textrm{if }x_1, x_2\in A_s\\
1+\beta_s\vee \beta_{s'} & ,\,\textrm{if }x_1\in A_s, x_2\in A_{s'}.\\
\end{array}
\right.
\end{eqnarray}
The expression of $\theta(x_1,x_2)$ suggests an estimation method for the model constants $\beta_i$, $i=1,\dots,k$. \\
For each $i$, if we choose two locations $x_1$ and $x_2$ in $A_i$, we have
\begin{eqnarray}\nonumber
\begin{array}{rl}
\displaystyle \widehat{\beta}_i=\widehat{\theta}(x_1,x_2)-1,
\end{array}
\end{eqnarray}
where $\widehat{\theta}(x_1,x_2)$ can be obtained as we proposed in Section \ref{sestimation}.\\
From the expression 
\begin{eqnarray}\nonumber
\begin{array}{rl}
\displaystyle \theta(V(x))=1+\sum_{y\in V(x)}\beta(y)-\bigwedge_{y\in V(x)}\beta(y),
\end{array}
\end{eqnarray}
we can also conclude that, in this model, $\theta(V(x))\in]1,|V(x)|]$, thus excluding total dependence.\\
For $A\subset \mathbb{Z}^2$, by applying the proposition 3.1, we obtain
\begin{eqnarray}\nonumber
\begin{array}{rl}
\displaystyle \zeta(A)=&\displaystyle\frac{\nu(A)-\sum_{x \in A}\left(\sum_{y\in V(x)}\beta(y)+\bigvee_{y\in V(x)}(1-\beta(y))\right)}{\nu(A)-|A|}\\\\
=&1-\displaystyle\frac{\sum_{x \in A}\sum_{y\in V(x)}\beta(y)-\sum_{x \in A}\bigwedge_{y\in V(x)}\beta(y)}{\nu(A)-|A|}.
\end{array}
\end{eqnarray}
Consider, to easily illustrate this crossinggram, regions $A\subset A_i$ and such that $V(x)\subset A_i$ for each $x\in A$. Then, for these regions, we get $\theta(V(x))=1+(|V(x)|-1)\beta_i$ and $\zeta(A)=1-\beta_i$.
Therefore, the lower the beta the greather the dependence between variables on A ( smaller $\theta$) and the greather the crossinggram value
, which indicates a smaller proportion of upcrossings among exceedances of high levels.
This result agrees with what we would expect from the definition of the random field since lower $\beta$ value potentiates the leveling effect of the factor $R$.\\
We simulate this random field, for $\mathcal{P}= \cup_{i=1}^3 A_i$, with $A_i=\{(x,y) \in \mathbb{Z}^2: d_{i-1}^2 \le x^2 + y^2 < d_i^2\}$, $i=1,2,3$, $d_0=0,$ $d_1=12,$ $d_2=34,$ $d_3=+\infty$, $\beta_1=8/10$,  $\beta_2=6/10$, $\beta_3=1/10$, over $C=\{(x,y) \in \mathbb{Z}^2:  x^2 + y^2 < 50^2\}$.
The generated trajectory can be seen in Fig.~\ref{fig:1}. Despite the plots can't provide a quantitative information for the intensity of local upcrossings we can see over $A_1\cap C$ a more rugged trajectory, corresponding to a $\zeta(A_1\cap C)=\frac{2}{10}$, and a smoother trajectory over $A_3\cap C$ corresponding to $\zeta(A_3\cap C)=\frac{9}{10}$.

\begin{figure}
 \centering
 \includegraphics[width=17cm]{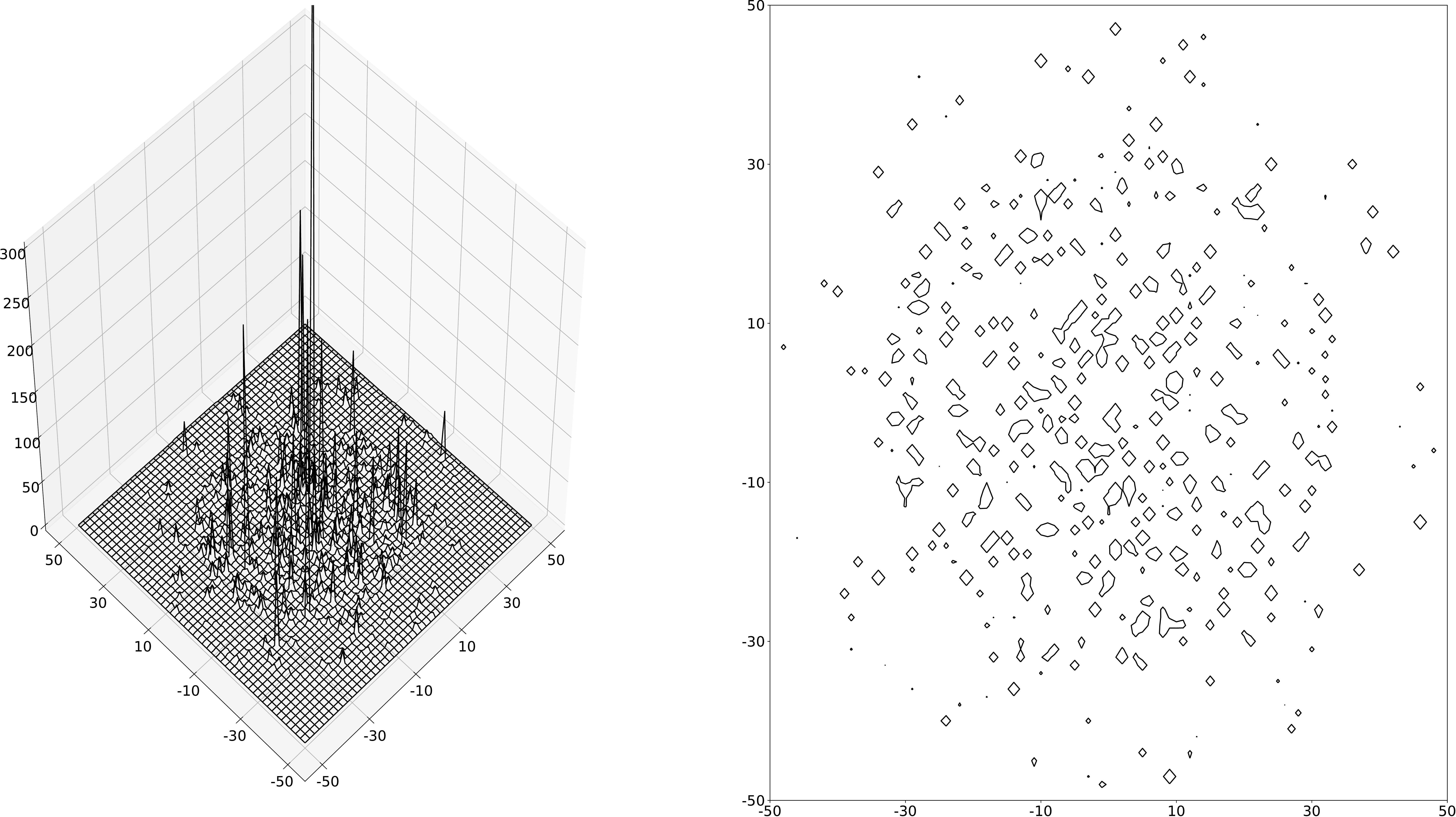}
 \label{fig:1}
 \caption{The simulated random field over $C=\{(x,y) \in \mathbb{Z}^2:  x^2 + y^2 \le 50^2\}$.} 
\end{figure}

\section{Conclusion}
In the trajectories of a random field on a lattice, the propensity for oscillations, meaning the proportion of exceedances of high levels which are upcrossings,  is inversely related to the degree of dependence and concordance between the random variables that generate it. We intended to quantify this propensity through coefficients that are easy to estimate and use in applications.
We defined a crossinggram for max-stable random fields on lattices, which take values in [0,1] and are larger the more dependent and concordant  the variables in the field are.
These coefficients are related to the extremal coefficients usually found in the literature of extreme values. They also have a representation from the expected values of local maxima of the random field. This representation motivates the estimation method  proposed and applied.
The coefficients range from 0 to 1, where $0$ represents a very local rough random field and  $1$  maximum local smoothness.
The proposed estimator has Normal asymptotic distribution and can be used in practical applications.
The proposed coefficients give good insight into the propensity for upcrossings by a max-stable random field from the theoretical point of view and in the example considered. They are easy to estimate and can be widely used.

\end{document}